\documentclass[12pt]{amsart}
\usepackage{amssymb,amsmath,amsfonts,amsthm}
\usepackage{mathrsfs}
\usepackage{mathtools}
\usepackage[dvipsnames]{xcolor}
\usepackage{graphicx}
\usepackage{float}
\usepackage{subcaption}
\usepackage{sidecap}
\usepackage[font=small, labelfont=bf]{caption}
\usepackage{enumitem}
\usepackage[T1]{fontenc}
\usepackage[margin=1.0in]{geometry}
\usepackage[colorlinks,citecolor=MidnightBlue,urlcolor=MidnightBlue,
linkcolor=RoyalBlue]{hyperref}

\usepackage{fancyhdr}
\newtheorem{prop}{Proposition}[section]
\newtheorem{theorem}{Theorem}[section]

\begin{document}

\title{General order adjusted Edgeworth expansions for generalized
  $t$-tests}
\author{Inna Gerlovina$^{1, 2}$}
\author{Alan E. Hubbard$^1$}
\email{innager@berkeley.edu}
\maketitle

\vspace{-3mm}
\begin{center}
  {\small
  $^1$\textit{Division of Biostatistics, University of California,
  Berkeley}\\
$^2$\textit{EPPIcenter, University of California, San Francisco}}
\end{center}

\begin{abstract}
  We develop generalized approach to obtaining Edgeworth expansions 
  for $t$-statistics of an arbitrary order using computer algebra and
  combinatorial algorithms. To incorporate various versions of
  mean-based statistics, we introduce Adjusted Edgeworth expansions
  that allow polynomials in the terms to depend on a sample size in a
  specific way and prove their validity. Provided results up to 5th
  order include one and two-sample ordinary $t$-statistics with biased
  and unbiased variance estimators, Welch $t$-test, and moderated
  $t$-statistics based on empirical Bayes method, as well as general
  results for any statistic with available moments of the sampling
  distribution. These results are included in a software package that
  aims to reach a broad community of researchers and serve to improve
  inference in a wide variety of analytical procedures; practical
  considerations of using such expansions are discussed. 
\end{abstract}

\section{Introduction}

Higher-order asymptotics, and especially developments based on
Edgeworth expansions (EE), played an important role in statistical
inference for over a century - in particular as a means to obtain more
accurate approximation to the distribution of interest, to gain
understanding and establish properties of methods like bootstrap, 
and to compare different statistical procedures. While interest to
asymptotic expansions has been sustained throughout much of this time,
some advances in statistical theory and methodology brought renewed
attention to EE - such as fundamental theoretical results of
Bhattacharya and Ghosh \cite{bhattacharyaGhosh1978,
  bhattacharyaRao1986book1976} and introduction of bootstrap
\cite{efron1979bootstrap}. More recently, proliferation of massive
amounts of data, often with complicated structure, introduced specific
challenges where higher-order inference procedures could be very
beneficial - for example, small sample size and high-dimensional data
analysis that requires probability estimation in far tail regions as a
consequence of some multiple testing procedure. For these challenges,
EE might offer a promising direction and become a widely used
practical tool. \\ 

Tremendous amount of research has been conducted on validity and
derivation of EE for many tests, classes of estimators, and test
statistics. Among them, to name just a few, are Hotelling $T^2$ test
\cite{kano1995, fujikoshi1997}, linear and non-linear regression
models \cite{qumsiyeh1990, lahiri1996}, Cox regression model
\cite{gu1992}, linear rank statistics \cite{albersBickel1976, bickelZwet1978,
skovgaard1986MV, kallenberg1993}, M-estimators \cite{lahiri1996}, and
 U-statistics \cite{bickel1986edgeworth, callaert1980,
   helmers1991}. Expansions have been used for Bayesian methods 
 (e.g. posterior densities) \cite{kolassa2020bayes}, random trees
 \cite{kabluchko2017trees}, permutation tests \cite{yang2019perm}, and
  sampling procedures \cite{yousef2014}. They have been developed for   
various dependent data structures: Markov chains
\cite{bertail2004markov}, martingales \cite{mykland1993},
autoregression and ARMA processes
\cite{taniguchi1987,kakizawa1999AR,mikusheva2015auto}. Some papers
focus specifically on multivariate analysis
\cite{skovgaard1986MV,fujikoshi1997,lahiri1996,andersonHall1998HD}. Applications
of EE range from physics, astronomy, and signal processing to finance,
differential privacy, design optimization, and survey sampling. \\

Research establishing validity and theoretical properties of
asymptotic expansions, starting with Cramer
\cite{cramer1928composition}, has been the basis for developing
EE. Classical Edgeworth expansion theory regarded a sum of independent
identically distributed variables - standardized sample mean
(i.e. scaled by its known standard deviation). This was followed by
work of Petrov \cite{petrov2012book} that proved the results for sums
of independent but not necessarily identically distributed random
variables; later research extended EE to sums of independent and
somewhat dependent random variables,
e.g. \cite{skovgaard1986MV}. However, in order to use EE as an  
inferential tool, expansions for studentized, not standardized,
statistics are needed as the variance is not normally known in 
practice and needs to be estimated - with $t$-statistic being the most
important and commonly used one. First expansions for a studentized
mean were derived by Chung \cite{chung1946} and included a fourth
order (3-term) expansion. Groundbreaking research by Bhattacharya and
Ghosh \cite{bhattacharyaGhosh1978} proved the validity of EE for any
multivariate asymptotically linear estimator in a general case. Their
moment conditions for studentized mean required finite $2 (k + 2)$
moments for a $k$-term expansion. Next important development for
$t$-statistic happened in 1987 when P. Hall introduced a special
streamlined way of deriving EE specifically for an ordinary
$t$-statistic, obtaining an explicit 2-term expansion \cite{hall1987edgeworth,
hall2013book}. He proved the validity of a $k$-term EE for
minimal moment conditions: $k + 2$ finite moments, which is exactly
the number of moments needed for the expansion, with a non-singularity
condition on an original distribution. Work that followed was
concerned with less resrictive (and later optimal) smoothness
conditions in various cases as well as results on Cramer condition 
\cite{bloznelis1999, baiRao1991, babu1993} and different dependence
conditions (most generally in \cite{lahiri2010}). \\

For many statistics and more general classes/groups of estimators, EE
are presented in a general form, often in terms of cumulants of the
distribution of the estimator or some intermediate statistics. As
such, they are not immediately adaptable for practical implementation,
which would require additional steps. These steps can include further
analytical processing, numerical methods such as numerical
differentiation, or estimation of the cumulants of sampling
distribution with the help of resampling methods such as jackknife
\cite{putter1998empirical} and Monte Carlo simulation
\cite{hall1999MC}. Conversely, expansions for $t$-statistic presented 
by P. Hall \cite{hall1987edgeworth} are expressed in terms of
cumulants of the original distribution (equal to standardized
cumulants since unit variance is assumed) and standard normal
p.d.f. This is the classical form of EE for the sum/mean; as exact
algebraic expressions, such ready-to-use expansions can be
incorporated into statistical analysis directly. \\  

In statistical inference, variance estimation is crucial, which makes it
a focus of various methods designed for specific assumptions and data
structures. By generalizing this part of studentized mean-based
statistics, we can provide higher-order inference to many data
analysis scenarios. Some common examples are na\"ive biased and 
efficient unbiased estimators, multi-sample estimators with and
without equality assumption, and shrinkage estimators.
When sample size $n$ is small or moderate, the difference between
unbiased $s^2_{unb} = (n - 1)^{-1}\sum_{i = 1}^n {\left(X_i - \bar{X}
\right)}^2$ and biased $s_b^2 = n^{-1}\sum_{i = 1}^n {\left( X_i -
\bar{X} \right)}^2$ variance estimates is not
negligible. Historically, most expansions for $t$-statistics were 
developed for the biased estimator; Chung \cite{chung1946} 
mentions the unbiased version before switching to the biased one ``for
brevity'', Hendriks at al \cite{hendriksKlaassen2006} consider $s_{unb}^2$
and suggest an approximated correction for it based on Taylor
expansion. With generalized one- and two-sample EE, we are able to
incorporate all of these variants including pooled variance for a
two-sample $t$-statistic and posterior variance used in moderated
$t$ based on empirical Bayes method \cite{smyth2004statistical}, which
provides more stable inference in high-dimensional data analysis,
especially when the sample size is small. \\

To incorporate various estimators into a generalized framework and
to simplify results and make them readily available for practical use, we
propose adjusted Edgeworth expansions (AEE) that allow certain sample
size dependent coefficients to stay unexpanded throughout the process of
derivation and carry through to the results - and prove AEE's validity
as asymptotic expansions. We derive closed form general order
expressions for moments of sampling distribution; using these
expressions, software algorithms, and computer algebra, arbitrary
order expansions can be generated and used for practical
applications. Throughout the paper, we adopt the terminology of ($k +
1$)'th order or $k$-term expansion, where normal approximation is a
zero term. In most of the literature, expansions are derived up to the
second or third order (Chung presents $3$-term or fourth order
expansion for an ordinary one-sample $t$-statistic). With small
samples and distributions that are far enough from Gaussian,
especially highly skewed distributions, closer approximations and
terms beyond second or third order might be desirable. Other benefits
of having subsequent terms include insights into the error of the
approximation or comparisons between different procedures based on
lower order approximations \cite{bickel1974review}. For a general
order EE for standardized mean, which is the original classical case
for EE (sum of independent random variables), Blinnikov et al
\cite{blinnikov1998} proposed a software algorithm and calculated $12$
terms; such expansions also fit into our generalized version as a
special case. We provide results up to fifth order for one- and
two-sample $t$-tests
(\href{https://github.com/innager/AEE}{Supplementary materials} and R
package \href{https://github.com/innager/edgee}{\textit{edgee}}
\cite{packageEdgee}); $4$-term AEE for the simplest case of one-sample
ordinary $t$-statistic is presented in the main text. \\

This paper is organized as follows: section \ref{sec:roadmap} outlines
a roadmap to derive the expansions and introduces AEE; it is followed
by one- and two-sample expressions for general order moments of
sampling distributions (first step in the roadmap). Section
\ref{sec:theorem} establishes validity of AEE; in section
\ref{sec:results}, we provide general results along with examples of
specific cases including ordinary one- and two-sample
$t$-statistics, Welch $t$-test, and moderated $t$-statistics
calculated with posterior variance in high-demensional data
analysis. Illustrations for higher-order approximations based on
expansions of different orders are provided in section
\ref{sec:illustr}. We conclude with a discussion on specific features 
of AEE for studentized means and considerations for their practical
applications.   \\

\section{Adjusted Edgeworth Expansions} \label{sec:roadmap}  

With the goal of generating an arbitrary order explicit expansions
expressed in terms of cumulants or central moments of the data
generating distribution for a generalized mean-based statistic, we
review and modify the steps of the roadmap for obtaining EE. In
general, for some test statistic $\hat{\theta}$, these steps include
Taylor expansion of $\hat{\theta}$'s characteristic function,
collecting the terms according to the powers of sample size $n$,
truncating the expression to the desired order, and using Hermite
polynomials to get EE through inverse Fourier transform. Two steps
in particular are the focus of our approach: 1. the first step in the
process, deriving cumulants of the sampling distribution, which is the
part tailored to the specific test statistic, and 2. collecting the
terms by the powers of $n$ with subsequent truncation. Since the
ultimate goal of this work is producing higher-order expansions
suitable for practical applications in a wide class of statistical tests,
considerations of manageability of results and feasibility of derivations
play an important role in this approach.  \\ 

Coefficients in EE get progressively longer and harder to
obtain with each additional term. Prior to the use of computer
algebra, expansions of very limited orders have been derived in their
explicit form - even for a basic statistic such as sample
average. Computer algebra and software algorithms allow one to handle
long calculations and generate expressions for high orders that were
previously challenging and prone to human errors. Moreover, generated  
results can be used in data analysis directly as source code, further 
automating the process. Most of the steps following the initial
derivation of cumulants of the sampling distribution are
straightforward; the step with collecting and truncating terms with
respect to sample size $n$, however, deserves special attention and
will be addressed separately. \\ 

Let $X_1, \dotsc, X_n$ be a sample of $n$ i.i.d. random variables with
central moments $\mu_j$ and let $\hat{\theta}$ be some
normalized test statistic with
c.d.f. $F_{\hat{\theta}}(\cdot)$. Consider a $K$-term Edgeworth
expansion $F_{n, K}$ of $F_{\hat{\theta}}$:     
\begin{equation*} 
 F_{n, K}(x) = \Phi(x) + n^{-\frac{1}{2}} q_1(x)\phi(x) + n^{-1}q_2(x) \phi(x)
+ \dotsb + n^{-\frac{K}{2}} q_K(x) \phi(x), 
\end{equation*}
where polynomials $q_i(x)$ are written in terms of $\mu_j$ or
standardized cumulants $\lambda_j$ and do not depend on $n$;
$\Phi(\cdot)$ and $\phi(\cdot)$ denote standard normal c.d.f. and
p.d.f. respectively. For a two-sample or multiple-sample test
statistic, this expression should either be modified to incorporate 
sample sizes $n_1, n_2, \dotsc$, or some summary measure $n$ can be
used to conform to the above representation. \\

Sample size $n$ is a key component in all variance estimators; viewed
as a function of $n$, each estimator has a different functional
form. In order to obtain EE for a generalized $t$-statistic, we need
to come up with a form that would encompass many of the useful
estimators. This approach, however, would become an obstacle for
creating a power series in $n^{-1/2}$ - a ``collecting'' step, which calls
for arranging the terms based on powers of $n$. On the other hand,
if we were to attempt obtaining EE for each individual case (without
generalization), all the versions except the na\"ive biased variance
estimator $s_b^2 = n^{-1} \sum_{i = 1}^n (X_i - \bar{X})^2$ would yield
such prohibitively complicated coefficients in this power series that
both derivation and use would become unfeasible after the first few
orders. \\

To address the challenges posed above, we introduce adjusted Edgeworth
expansions (AEE) that would simplify the results and allow a
generalized solution for different kinds of $t$-tests (or other types
of studentized statistic). For a generalized variance estimator,
consider a set of coefficients that depend on $n$ and satisfy certain
order conditions (e.g. $C = \sum_{i = 0}^{\infty} c_i \,
n^{-i}$) but whose functional form is specific to the estimator. In
AEE, the collecting and truncating steps will leave these coefficients
intact, carrying them over to the results, which thus remain
generalized. Leaving the coefficients unexpanded (as functions of $n$)
leads to their presence in the characteristic function and therefore
requires a subsequent adjustment to the inverse Fourier transform
step. In that step, the term $(it)^2/2$ that in classic EE becomes a
standard normal c.d.f. in expansion's zero term (recall that
$\hat{\theta}$ is a normalized statistic) now aquires another factor,
which means that a normal c.d.f. in zero term is no longer standard -
and that its variance depends on $n$. Let $r^2$ denote this factor and
call it variance adjustment; $r^2 \to 1$ as $n \to \infty$. Then for a
term $k$ we get   
\begin{equation*} 
 (-it)^k e^{-\frac{1}{2} t^2 r^2} = \frac{1}{r^k}
 \int_{-\infty}^{\infty} e^{itx} \phi^{(k)} \left(\frac{x}{r} \right)
 dx 
   = (-1)^k \, \frac{1}{r^k} \int_{-\infty}^{\infty} e^{itx} \, He_k
   \left(\frac{x}{r} \right) \phi \left(\frac{x}{r} \right) dx, 
\end{equation*}
where $\phi^{(k)} \! \left(\frac{x}{r} \right) = \frac{d^k}{dy^k} \,
\phi (y) \Big|_{y = \frac{x}{r}}$ and $He_k(x) = (-1)^k \,
e^{\frac{x^2}{2}} \dfrac{d^k}{dx^k} \, e^{-\frac{x^2}{2}}$ are Hermite
polynomials. Therefore in inverse Fourier transform $r^{-k} \, He_{k -1 }
\left(\frac{x}{r} \right)$ will be substituted for $(it)^k$.  \\

It follows that AEE can themselves be viewed as a generalization on
EE, with coefficients for classic EE being constants (not depending on
$n$). Let 
\begin{equation} \label{eq:AEE}
 \tilde{F}_{n, K}(x) = \Phi\left(\frac{x}{r}\right) +
 n^{-\frac{1}{2}} \, q_1(x; r) \, \phi\left(\frac{x}{r}\right) + \dotsb +
 n^{-\frac{K}{2}} \, q_K(x; r) \, \phi\left(\frac{x}{r}\right)
\end{equation}
be a $K$-term AEE of $F_{\hat{\theta}}(x)$. When $r^2 = 1$, it is a classic
EE; that is the case with one-sample $t$-statistic with variance
estimator $s_b^2$ and two-sample statistic for a Welch $t$-test with
na\"ive biased estimators for both samples. Asymptotic expansion
property of general case EE has been long established 
(\cite{bhattacharyaGhosh1978}):
$F_{\hat{\theta}}(x) - F_{n, K}(x) = o \left(n^{K/2}\right)$ but it
does not apply to AEE in general. With specific order conditions that are
satisfied by most mean-based test statistics, we extend this result
and establish validity of AEE for $t$-tests (section \ref{sec:theorem}).  \\

\section{Moments of Sampling Distribution} \label{sec:moments}

The first step in the roadmap - deriving general order closed form
expressions for non-central moments of the sampling distribution, from 
which the cumulants are easily obtained. For a $K$-term EE, only a
limited number of terms in cumulants of a sampling distribution is
used; terms that correspond to orders of $n^{-(K + 1)/2}$ or higher
are truncated. An important consideration for generating these
cumulants is computational efficiency and feasibility of algebraic
manipulation of long expressions - this consideration motivates the
form for the moments that we present, which avoids generating
unnecessary terms in the first place. The framework that generalizes
$t$-tests of various kinds introduces variables $A$ and $B$ that
depend on $n$ in a certain way and are specific to the variance
estimators.  \\ 

 Let $X$ be a random variable with $E(X) = 0$ (we can consider a
mean-zero random variable without any loss of generality), variance
$\sigma^2 = \mathcal{O}(1)$, central moments $\mu_j$,
and standardized cumulants $\lambda_j = \kappa_j/\sigma^j$,
where $\kappa_j$ is a $j$'th cumulant; let $X_1, \dotsc, X_n$ be a random
sample as in section \ref{sec:roadmap}. We also use the following
notation: $\bar{X} = n^{-1} \sum_{i = 1}^n X_i = \mathcal{O} \left(
  n^{-1/2} \right)$, $\overline{X^2} = n^{-1} \sum_{i = 1}^n X_i^2 =
\mathcal{O} (1)$, and $\bar{X}_s = n^{-1} \sum_{i = 1}^n 
\big(X_i^2 - \sigma^2 \big) = \overline{X^2} - \sigma^2 = \mathcal{O}
\left( n^{-1/2} \right)$. Let $s^2$ be some estimator of $Var(X)$ that
can be written as $s^2 = A + B(\bar{X}_s - \bar{X}^2)$, where $A > 0$,
$B > 0$; a corresponding estimator of $Var(\bar{X})$ is then
$s^2_{\bar{X}} = s^2/n$. Consider a statistic of the form   
\begin{equation} \label{eq:theta1}
  \hat{\theta} = \frac{\bar{X}}{s_{\bar{X}}} = \frac{\sqrt{n} \,
    \bar{X}}{s} = n^{\frac{1}{2}} \bar{X} \left[ A + B
   \left(\bar{X}_s - \bar{X}^2\right)\right]^{-\frac{1}{2}} 
\end{equation} 

\begin{prop} \label{prop:mom1}
 $m$'th order moments of sampling distribution of $\hat{\theta}$
 defined in \eqref{eq:theta1} are given by
 \begin{align*} 
 \mu_{\hat{\theta}, m} = E \left[ \hat{\theta}^m \right] = n^{\frac{m}{2}}
 A^{-\frac{m}{2}} \Bigg[ \rho(m, 0) + \sum_{k = 1}^K &\sum_{i =
0}^{\left\lfloor \frac{k}{2} \right\rfloor} a_{m, k - i} (-1)^i {k -
i \choose i} \bigg(\frac{B}{A}\bigg)^{k - i} \notag \\
& \; \times \rho(m + 2i, k - 2i)
\Bigg] + \mathcal{O} \left(n^{-\frac{K + 1}{2}} \right),
\end{align*}
where 
 \begin{equation} \label{eq:amk}
 a_{m,k} = \frac{1}{k! \, 2^k} (-1)^k \prod_{j=0}^{k-1} (m + 2 \, j)
\end{equation}
  and $\rho(i, j) = E \left( \bar{X}^i \, \bar{X}_s^j \right)$.
\end{prop}
 The proof of this Proposition is provided in Appendix \ref{sec:proofProp}. \\

Let
\begin{equation*} 
 \nu(k, l) = E \left[ \bar{X}^k \left( \overline{X^2} \right)^l
 \right] = \frac{1}{n^{k + l}} \sum_{i_1=1}^n \dotsm \sum_{i_k=1}^n
   \sum_{j_1=1}^n \dotsm \sum_{j_l=1}^n E ( X_{i_1} \dotsm  X_{i_k}
   X_{j_1}^2 \dotsm X_{j_l}^2); 
 \end{equation*}
then
\begin{equation*} 
 \rho(i, j) = \sum_{k=0}^j (-1)^k {j \choose k} \sigma^{2k} \nu(i, j-k). 
 \end{equation*}
 $\nu_{k, l}$ is a special case of expectation $E\big(
 \overline{X^{\phantom{.}}}^{j_1} \overline{X^2}^{j_2} \dotsm
 \overline{X^m}^{j_m} \big)$ for an arbitrary $m$, expression for
 which in terms of $\mu_j$ and $n$ can be generated using a
 combinatorial algorithm described in \cite{gerlovina2019moments} with
 an R package \textit{Umoments}~\cite{packageUmoments}. \\

For a $K$-term EE for mean-based statistics,
we need to find $E(\hat{\theta}^m)$, $m = 1, \dots, M$, where $M = K +
2$. For other statistics, that might involve a different number of
moments of the original distribution - for example, $K$-term 
expansion for a sample variance will require $2(K + 2)$
moments/cumulants. Figure \ref{fig:klgrid} shows which $\nu(k, l)$
are needed for different orders of Edgeworth expansions. For
straightforward calculation of $\hat{\theta}^m$ that is based on 
equation \eqref{eq:theta1}, we would expand $\left[ 1 + B/A \, (\bar{X}_s
  - \bar{X}^2) \right]^{-\frac{m}{2}}$ and subsequently substitute
$K$ for $\infty$ in the sum limit (see also \eqref{eq:thetamt1}); the
set $\{(k, l)\}$ required for this approach would be represented by a
rectangle. By rearranging the terms and grouping them with respect to
$n$, we have cut out the area in the bottom right corner; even
though the number of expressions in that corner is comparatively small,
these expressions are much longer than the ones in the rest of the rectange.
For example, excluding
these expressions reduces the time to generate a set of $\nu(k,
  l)$ (shaded area vs rectangular grid) by factors of $100$ for
$K = 4$, $900$ for $K = 5$, and $6000$ for $K = 6$. \\

\begin{SCfigure} 
 \includegraphics[scale=0.48]{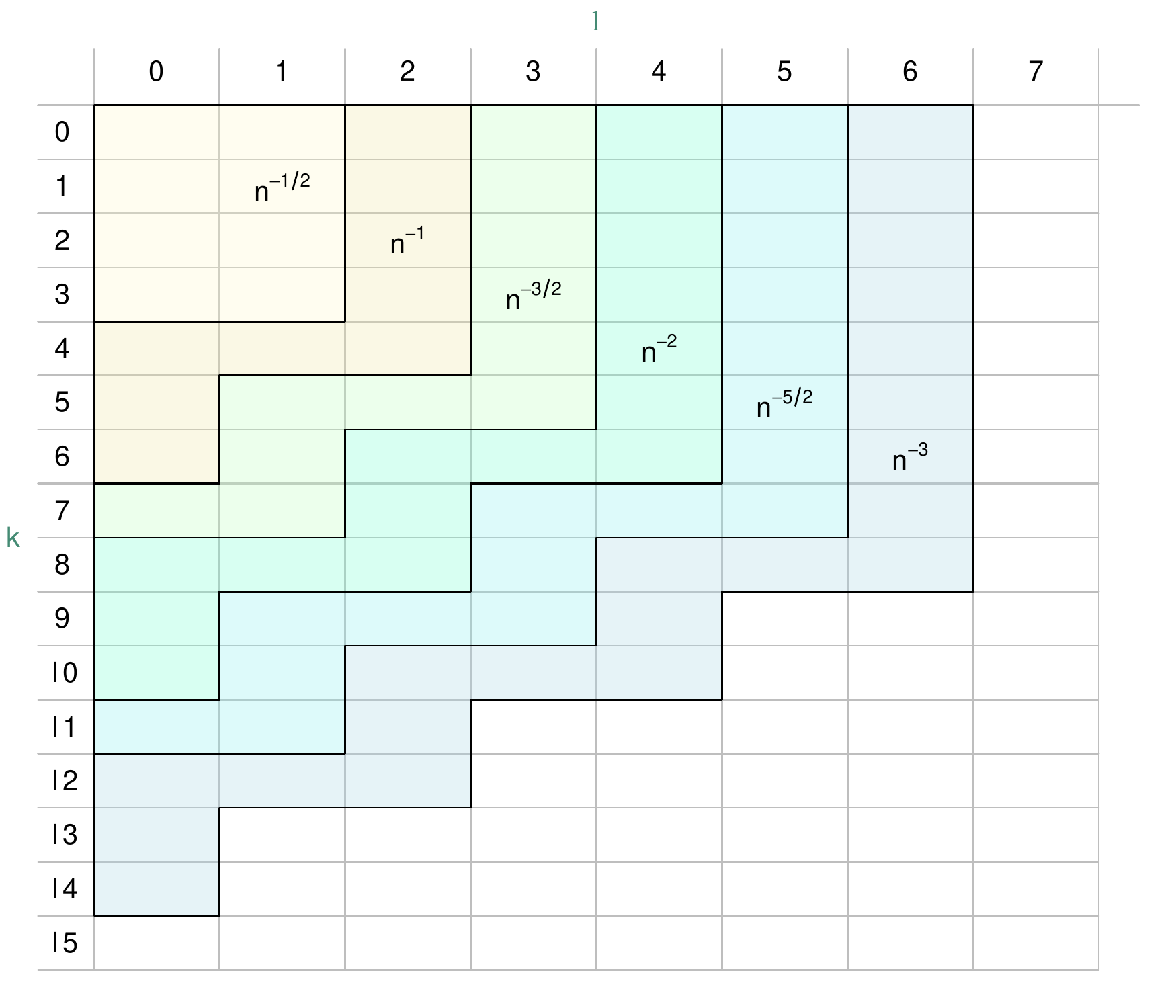}
 \caption{The grid showing which $\nu(k, l)$ need to be caculated for
   various terms of Edgeworth expansion for a $t$-statistic, with
   terms indicated as powers of $n$. Combinations of $k$ (rows) and
   $l$ (columns) needed for a particular term also include all the
   combinations needed for previous terms as well.}
 \label{fig:klgrid}
\end{SCfigure}

\begin{sloppypar}
For a generalized two-sample $t$-test, consider mean-zero random
variables $X$ and $Y$ with variances $\sigma_x^2$ and $\sigma_y^2$
respectively, central moments $\mu_{x, j}$, $\mu_{y, j}$, and a random
sample $X_1, \dotsc, X_{n_x}, Y_1, \dotsc, Y_{n_y}$. Similarly to the
one-sample case, define $\bar{X} = n_x^{-1} \sum_{i = 1}^{n_x} X_i$,
$\bar{Y} = n_y^{-1} \sum_{i = 1}^{n_y}Y_i$, $\bar{X}_s = n_x^{-1} \sum_{i =
  1}^{n_x} (X_i^2 - \sigma_x^2)$, and $\bar{Y}_s = n_y^{-1} \sum_{i =
  1}^{n_y} (Y_i^2 - \sigma_y^2)$. As mentioned in section
\ref{sec:roadmap}, to have a single summary measure representing
sample size and to eliminate $n_x$ and $n_y$ (assuming they are
comparable), we introduce $n = (n_x + n_y)/2$, $b_x = n/n_x$, and $b_y
= n/n_y$. Let $s^2 = A + B_x(\bar{X}_s - \bar{X}^2) + B_y (\bar{Y}_s -
\bar{Y}^2)$ and let $s^2_{\bar{X} - \bar{Y}} = s^2/n$ be some estimator of 
$Var(\bar{X} - \bar{Y})$. In this case there is no immediate
interpretation for $s^2$ but it is a useful construct that is
analogous to the one-sample case. Consider a statistic of the form  
\end{sloppypar}
\begin{equation} \label{eq:theta2}
 \hat{\theta} = \frac{\bar{X} - \bar{Y}}{s_{\bar{X} - \bar{Y}}} =
 \frac{\sqrt{n} (\bar{X} - \bar{Y})}{s} = n^{\frac{1}{2}} (\bar{X} -
 \bar{Y}) \left[A + B_x(\bar{X}_s - \bar{X}^2) + B_y(\bar{Y}_s -
   \bar{Y}^2) \right]^{-\frac{1}{2}} 
\end{equation}

\begin{prop} \label{prop:mom2}
The moments of sampling distribution of $\hat{\theta}$ defined in
 \eqref{eq:theta2} are given by
 \begin{alignat*}{3} 
 E \left(\hat{\theta}^m \right) =  & \, n^{\frac{m}{2}} A^{-\frac{m}{2}}
 \sum_{j=0}^m (-1)^j {m \choose j} \bigg[ \rho(m-j, 0) \, \tau(j, 0)
 \nonumber \\  
  & \, + \sum_{k=1}^K
   \sum_{i=0}^{\left\lfloor \frac{k}{2} \right\rfloor} 
    (-1)^i a_{m, k - i} {k-i \choose i} A^{i-k} \sum_{u=0}^{k-2i} 
    \sum_{v=0}^i  && \! {k - 2i \choose u} \! {i \choose v} B_x^{(k -
      i) - (u + v)} B_y^{u + v} \nonumber \\
 & && \mkern-15mu \times \rho\big(m-j+2(i-v), k-2i-u\big) \,
 \tau(j+2v, u) \bigg] \nonumber \\
 & \, +  \mathcal{O} \left( n^{-\frac{K + 1}{2}} \right),  
\end{alignat*} 
where $a_{m, k}$ is the same as in \eqref{eq:amk}, $\rho(i, j) =
E(\bar{X}^i \bar{X}_s^j)$, and $\tau(i, j) = E(\bar{Y}^i
\bar{Y}_s^j)$. 
\end{prop}
The proof of this Proposition is in Appendix \ref{sec:proofProp}. \\

\section{Validity of AEE} \label{sec:theorem}

Let $X$ be a random variable with known moments and cumulants
$\kappa_j$; set $\kappa_1 = 0$, $\kappa_2 = 1$ without loss of
generality. Let $X_1, \dotsc, X_n$ be an i.i.d. sample. First,
consider a test statistic as in \eqref{eq:theta1} with a constraint
that $A$ and $B$ do not depend on $n$. The function $g(x, y) =
n^{\frac{1}{2}} x [A + B(y - x^2)]^{-\frac{1}{2}}$ is infinitely
differentiable in $x$ and $y$, so by the fundamental result of
Bhattacharia and Gosh \cite{bhattacharyaGhosh1978} and Hall
\cite{hall2013book}, if $X$ has sufficient number of finite moments,
there exists EE of the form
\begin{equation} \label{eq:EE1}
 P \left( \hat{\theta} < x\right) = F_{\hat{\theta}}(x) =
 \Phi_{A^{-1}}(x) + \sum_{k = 1}^K n^{-\frac{k}{2}} q_k (x; A, B) \phi_{A^{-1}}(x) + o
 \left(n^{-\frac{K}{2}} \right),
\end{equation}
where $q_k(x; A, B)$ are some polynomials in $x$ whose coefficients 
do not depend on $n$ and are expressed in terms of $A$ and
$B$. $\Phi_{A^{-1}}(\cdot)$ and $\phi_{A^{-1}}(\cdot)$ denote normal
$N(0, A^{-1})$ c.d.f. and p.d.f. In AEE, however, we consider test
statistics of the same form but with $A$ and $B$ replaced by $A_n$ and 
$B_n$, which do depend on $n$.

\begin{theorem} \label{thm:AEE1}
 Let $A_n = \sum_{j = 0}^{\infty} n^{-j}\tilde{a}_j$ and 
       $B_n = \sum_{j = 0}^{\infty} n^{-j} \tilde{b}_j$, where
       $\tilde{a}_j, \tilde{b}_j$ are constants and the series are
       absolutely convergent. Then, for a test statistic   
 \begin{equation*}
  \hat{\theta}_n = n^{\frac{1}{2}} \bar{X} \left[ A_n + B_n \left( \bar{X}_s -
      \bar{X}^2 \right)\right]^{-\frac{1}{2}},
 \end{equation*}
 there exists AEE of the form
 \begin{equation} \label{eq:thm1}
  P \left(\hat{\theta}_n < x \right) =
  F_{\hat{\theta}_n}(x) = \Phi_{A_n^{-1}}(x) + \sum_{k = 1}^K n^{-\frac{k}{2}}
  q_k(x; \, A_n, B_n) \phi_{A_n^{-1}}(x) + o \left(n^{-\frac{K}{2}} \right).
 \end{equation}
\end{theorem}
Note that expressions for polynomials $q_k(x; \, A_n, B_n)$ are 
the same as those for $q_k(x; \, A, B)$ in \eqref{eq:EE1} apart from 
$A_n$, $B_n$ replacing $A$, $B$. \\

The proof (provided in Appendix \ref{sec:proofThm}) derives the order
of finite-term difference between two 
series that represent cumulants of $\hat{\theta}_n$: one that can be
used for classic EE and the other - for AEE. Consequently, using
the difference $F_{n, K}(x) - \tilde{F}_{n, K}(x)$ and validity of
classic EE, we establish the order of 
$F_{\hat{\theta}_n}(x) - \tilde{F}_{n, K}(x)$. \\

To explicitly relate \eqref{eq:thm1} to the original expression for
AEE \eqref{eq:AEE}, consider the case where $Var(X) = \sigma^2$ and
let $r^2 = \sigma^2/A_n$. Then, substituting $x' = x/\sigma$ for $x$
in \eqref{eq:thm1}, we get $\Phi_{A_n^{-1}}(x') = \Phi(x' \sqrt{A_n})
= \Phi(x'/r)$ and $\phi_{A_n^{-1}}(x) = \phi(x'/r)$. \\ 

For a two-sample $t$-test, as previously, we consider a sample $X_1,
\dotsc, X_{n_x}, Y_1, \dotsc, Y_{n_y}$ and set $n = (n_x + n_y)/2$.

\begin{theorem} \label{thm:AEE2}
  Let $A_n = \sum_{i = 0}^{\infty} \left(\tilde{a}_{xi}/n_x^i
    + \tilde{a}_{yi}/n_y^i \right)$,
  $B_{xn} = \sum_{i = 0}^{\infty} \tilde{b}_{xi}/n_x^i$, and
  $B_{yn} = \sum_{i = 0}^{\infty} \tilde{b}_{yi}/n_y^i$, where
  $\tilde{a}_{xi}$, $\tilde{a}_{yi}$, $\tilde{b}_{xi}$, and
  $\tilde{b}_{yi}$ do not depend on $n_x$, $n_y$, and the series are
  absolutely convergent. Then, for a test statistic
\begin{equation*}
  \hat{\theta}_n = n^{\frac{1}{2}} (\bar{X} - \bar{Y}) \left[A_n +
  B_{xn} (\bar{X}_s - \bar{X}^2) + B_{yn} (\bar{Y}_s - \bar{Y}^2)
\right]^{-\frac{1}{2}} 
\end{equation*}
there exists AEE of the form
\begin{equation*} 
  P \left(\hat{\theta}_n < x \right) =
  F_{\hat{\theta}_n}(x) = \Phi_{A_n^{-1}}(x) + \sum_{k = 1}^K n^{-\frac{k}{2}}
  q_k(x; \, A_n, B_{xn}, B_{yn}) \phi_{A_n^{-1}}(x) + o \left(n^{-\frac{K}{2}} \right).
\end{equation*}
\end{theorem}
The proof of this Theorem is in Appendix \ref{sec:proofThm}. \\

\section{Results} \label{sec:results}

In this section, we provide expressions for AEE at different levels of
generalization. Recall that in the process of obtaining EE, cumulants
$\kappa_{\hat{\theta}, j}$ of sampling distribution are expressed as
power series in $n^{-1/2}$. As seen in, for example,
\cite{hall2013book}, \cite{bickel1974review}: 
\begin{equation*}
 \kappa_{\hat{\theta}, j} = n^{-\frac{j-2}{2}} \left( k_{j, 1} + n^{-1} k_{j, 2} +
   n^{-2} k_{j, 3} + \dotsb \right),  \qquad j \geqslant 1. 
\end{equation*}
Once $k_{j, l}$ are obtained, they can be used to calculate
polynomials $q_k(x; r)$ in \eqref{eq:AEE}, together with Hermite
polynomials. Expressions for $q_k$ as functions of $k_{j, l}$ can be
used for AEE of any test statistic. Next, for one- and two-sample
generalized $t$-statistics, we look at $k_{j, l}$ as functions of
$\mu_j$, $A_n$, and $B_n$ (going forward, we omit the subscript $n$
for brevity). Finally, we provide $A$, $B$, and $r^2$ for some
commonly used versions of these statistics as well as for moderated
$t$-statistics based on empirical Bayes methods
\cite{smyth2004statistical}). In addition, for the simplest special
case of an ordinary one-sample $t$-statistic with na\"ive biased and
unbiased variance estimators, this nested chain of expressions reduces
to a nice short form where $q_k(x; r)$ are given in terms of
standardized cumulants $\lambda_j$ (provided here for the 4-term
AEE). For these particular statistics, such form is useful for
calculations and also allows for an illuminating comparison with known 
expressions for standardized mean. 2-term EE of this kind is found
in \cite{hall1987edgeworth, hall2013book}. \\

\subsection{General case}

For a given test statistic, first few polynomials $q_k(x; r)$ of AEE
\eqref{eq:AEE} are given by
\begin{flalign*}
 q_1(x; r) &= - \frac{1}{6 \, r^3} \, k_{3,1} \, He_2 \left(\frac{x}{r}
 \right) - \frac{1}{r} \, k_{1,2} &
\end{flalign*}
\begin{flalign*}
 q_2(x; r) = &- \frac{1}{72 \, r^6} \, k_{3,1}^{2} \, He_5\left(\frac{x}{r}\right)
 - \frac{1}{24\, r^4} \, {\left(4 \, k_{1,2} k_{3,1} + k_{4,1}\right)} \,
 {He}_3 \left(\frac{x}{r}\right) \notag \\[1mm]
               & - \frac{1}{2 \, r^2} \, {\left(k_{1,2}^{2} +
                   k_{2,2}\right)} \, {He}_1 \left(\frac{x}{r}\right), &
\end{flalign*}
where $He_j(x)$ are probabilists' Hermite polynomials. Since $k_{j,
  l}$'s do not depend on $x$, this approach is especially useful if
$\tilde{F}_{n, K}(x)$ needs to be calculated for many values of $x$. \\

For generalized one- and two-sample $t$-statistics, we show some lower  
order $k_{j, l}$'s in this paper; all $k_{j, l}$'s needed for fifth
order AEE, as well as remaining general case $q_k(x; r)$, can be found 
in the \href{https://github.com/innager/AEE}{\textit{Sage} notebook}
and \href{https://github.com/innager/edgee}{\textit{edgee} R
  package} \cite{packageEdgee}. Note that $k_{2, 1} = r^2$. \\

For the one-sample $t$-statistic: 
\begin{flalign*}
  &k_{1, 2} = - \frac{B \mu_3}{2 A^{\frac{3}{2}}}&
\end{flalign*}
\begin{flalign*}
  &k_{1, 3} = -\frac{6(8\mu_2 \mu_3 - \mu_5)A B^2 - 15 (\mu_2^2 \mu_3
    - \mu_3 \mu_4) B^3 - 8 A^2 B \mu_3}{16 A^{\frac{7}{2}}} & 
\end{flalign*}
\begin{flalign*}
  &k_{2, 1} = \frac{\mu_2}{A}&
\end{flalign*}
\begin{flalign*}
  &k_{2, 2} = \frac{4(4\mu_2^2 - \mu_4)AB - (4\mu_2^3 - 7
    \mu_3^2 - 4 \mu_2 \mu_4) B^2}{4 A^3}&
\end{flalign*}
\begin{flalign*}
  &k_{3, 1} = -\frac{(3 B \mu_2 - A) \mu_3}{A^{\frac{5}{2}}}& 
\end{flalign*}
\begin{flalign*}
  &k_{4, 1} = -\frac{(3 \mu_2^2 - \mu_4) A^2 - 6 (3 \mu_2^3 - \mu_3^2
    - \mu_2 \mu_4) A B + 3 (\mu_2^4 - 6 \mu_2 \mu_3^2 - \mu_2^2 \mu_4)
    B^2}{A^4} &
\end{flalign*}
\\
The two-sample $t$-statistic:
\begin{flalign*}
 &k_{1, 2} = -\frac{B_x b_x \mu_{x, 3} - B_y b_y \mu_{y, 3}}{2 A^{\frac{3}{2}}}&
\end{flalign*}
\begin{flalign*}
  &k_{2, 1} = \frac{b_x \mu_{x, 2} + b_y \mu_{y, 2}}{A}&
\end{flalign*}
\begin{flalign*}
  k_{2, 2} = &-\frac{4 B_x^2 b_x^2 \mu_{x, 2}^3 + 4 B_y^2 b_x b_y
    \mu_{x, 2} \mu_{y, 2}^2 + 4 B_y^2 b_y^2 \mu_{y, 2}^3 - 4 B_x^2
    b_x^2 \mu_{x, 2} \mu_{x, 4}}{4 A^3} \\
                &+ \frac{14 B_x B_y b_x b_y \mu_{x, 3} \mu_{y, 3} - 7
                  B_x^2 b_x^2 \mu_{x, 3}^2 - 7 B_y^2 b_y^2 \mu_{y, 3}^2}{4 A^3} \\
              & - \frac{4 A (B_x b_x^2 (4 \mu_{x, 2}^2 -\mu_{x, 4}) +
                B_y b_y^2 (4 \mu_{y, 2}^2 - \mu_{y, 4}) + (B_x b_x b_y + B_y b_x
     b_y) \mu_{x, 2} \mu_{y, 2})}{4 A^3} \\
              &+ \frac{4 (B_x^2 b_x b_y \mu_{x, 2}^2 -
     B_x^2 b_x b_y \mu_{x, 4}) \mu_{y, 2} - 4 (B_y^2 b_x b_y \mu_{x,
       2} + B_y^2 b_y^2 \mu_{y, 2}) \mu_{y, 4}}{4 A^3}&
\end{flalign*}
\begin{flalign*}
  k_{3, 1} = &- \frac{(3B_x b_x^2 \mu_{x, 2} + 3 B_x b_x b_y \mu_{y,
      2} - Ab_x^2) \mu_{x, 3} - (3B_y b_x b_y \mu_{x, 2} + 3B_y b_y^2
    \mu_{y, 2} - Ab_y^2) \mu_{y, 3}}{A^{\frac{5}{2}}} & 
\end{flalign*}
Thus $r^2 = \sigma^2/A$ for a one-sample $t$-statistic and $r^2 = (b_x
\sigma_x^2 + b_y \sigma_y^2)/A$ for a two-sample $t$-statistic. \\

\subsection{Examples of specific $t$-tests}

Statistics we consider here are a set of commonly used ordinary
$t$-statistics as well as moderated statistics that incorporate more
complex variance estimators \cite{smyth2004statistical}. For the first set, we
look at one-sample $t$-statistics with na\"ive biased and unbiased
variance estimators (with Bessel's correction), two-sample $t$-statistic
that assumes equal variances between two groups and uses pooled
(unbiased) variance estimator, and Welch $t$-tests that do not assume
equal variances - with both na\"ive biased and unbiased variance
estimators. For a higher-order approach to two-sample equal variance
test, we also assume equality of higher moments of distributions of
$X$ and $Y$. Note that these assumptions allow for a more efficient
estimator, so there is an advantage to using pooled variance if the
equality assumption is reasonable. \\ 

Moderated $t$-statistic,
which uses empirical Bayes approach, became a great practical tool
widely used in high-dimensional data analysis. In this case, the
normalizing factor is a posterior variance for a feature $g$ (e.g. a gene)
that incorporates prior information. The method uses a hierarchical
model, in which two hyperparameters $s_0^2$ and $d_0$ are estimated
from the data that has many features. Estimators for these
parameters have a closed form and are sufficiently stable due to the fact
that high dimensionality provides extensive information from which
only two hyperparameters are estimated - even when the number of
replicates (sample size) is small. This allows us to treat $s_0^2$ and
$d_0$ as constants in deriving AEE. 
Posterior variance $\tilde{s}_g^2$ for a feature $g$ is a linear
combination of $s_0^2$ and a sample/residual variance $s_g^2$: 
$\displaystyle{\tilde{s}_g^2 = \frac{d_0 s_0^2 + d_g {s_g^2}}{d_0 + d_g}}$, 
where $d_0$ and $d_g$ are prior and residual degrees of
freedom. Because of that, moderated $t$-statistic can also be viewed 
as a generalization for any scaled mean-based statistic as it can be
reduced to either standardized ($d_g = 0$) or studentized ($d_0 = 0$)
version. If data are distributed normally, moderated $t$-statistic
follows a $t$-distribution with augmented ($d_g + d_0$) degrees of
freedom.  \\

Let $C = \frac{n}{n - 1}$ for one-sample tests and $C_x =
\frac{n_x}{n_x - 1}$, $C_y = \frac{n_y}{n_y - 1}$, and $C_{xy} =
\frac{n_x + n_y}{n_x + n_y - 2} = \frac{n}{n - 1}$ for two-sample
tests (recall that $n = \frac{n_x + n_y}{2}$, $b_x = \frac{n}{n_x}$,
and $b_y = \frac{n}{n_y}$). Generalized expressions for variance
estimators $s^2 = A + B(\bar{X}_s - \bar{X}^2)$ (one-sample) and $s^2 =
A + B_x(\bar{X}_s - \bar{X}^2) + B_y(\bar{Y}_s - \bar{Y}^2)$ (two-sample)
allow us to easily extract $A$ and $B$ for each particular case. Note
that for moderated $t$-statistics $d_g = n - 1$ for one-sample and
$d_g = n_x + n_y - 2$ for two-sample tests (and thus $C_{xy} d_g = n_x
+ n_y = 2n$). \\

Table \ref{tbl:onesmp} provides expressions for $A$, $B$, and $r^2$
for various one-sample $t$-statistics, $\hat{\theta} = \sqrt{n}
\bar{X} / s$. For two-sample versions, $\hat{\theta} = \sqrt{n}
(\bar{X} - \bar{Y}) / s$, we introduce some short-hand notations: let
$s_x^2 = n^{-1} 
\sum_{i = 1}^{n_x} (X_i - \bar{X})^2$, $s_y^2 = n^{-1} \sum_{i =
  1}^{n_y} (Y_i - \bar{Y})^2$, and $s_{xy}^2 = (n_x + n_y)^{-1}
\big(\sum_{i = 1}^{n_x} (X_i - \bar{X})^2 + \sum_{i = 1}^{n_y} (Y_i
-\bar{Y})^2\big)$ be na\"ive biased estimators of $\sigma_x^2$,
$\sigma_y^2$, and $\sigma^2 = \sigma_x^2 = \sigma_y^2$
respectively. Table \ref{tbl:twosmp} shows expressions for $A$, $B_x$,
$B_y$, and $r^2$ for several two-sample $t$-tests. \\

\begin{table}[!h] 
 \centering
 \small
\renewcommand{\arraystretch}{3}
 \hspace*{-0.1cm}\begin{tabular}{ c p{1.7cm} c c c c } 
  type & variance estimator & $s^2$ & $A$ & $B$ & $r^2$ \\  \hline 
  ordinary & biased 
  & $\displaystyle{\frac{1}{n} \sum_{i = 1}^n (X_i - \bar{X})^2}$ 
  & $\sigma^2$ & $1$ & $1$ \\ \hline 
  ordinary & unbiased 
  & $\displaystyle{\frac{1}{n - 1} \sum_{i = 1}^n (X_i - \bar{X})^2}$ 
  & $C \sigma^2$& $C$ & $\displaystyle{\frac{1}{C}}$ \\ \hline
  moderated & posterior 
  & $\displaystyle{\frac{d_0 s_0^2 + \sum_{i = 1}^n (X_i - \bar{X})^2}{d_0 + n - 1}}$ 
  & $\displaystyle{\frac{d_0 s_0^2 + n \sigma^2}{d_0 + n - 1}}$ 
                          & $\displaystyle{\frac{n}{d_0 + n - 1}}$ 
                          & $\displaystyle{\frac{d_0 + n - 1}{d_0
                            s_0^2/\sigma^2 + n}} $ \\ \hline
  \end{tabular}
 \caption{One-sample statistics}
 \label{tbl:onesmp}
\end{table}

\begin{table} 
  \footnotesize
  \centering
\renewcommand{\arraystretch}{3}
 \hspace*{-1.4cm}\begin{tabular}{ c p{1.6cm} c c c c c } 
  type & variance estimator & $s^2$ & $A$ & $B_x$ & $B_y$ & $r^2$ \\ \hline 
  Welch & biased 
  & $b_x s_x^2 + b_y s_y^2$
  & $b_x \sigma_x^2 + b_y \sigma_y^2$
  & $b_x$ & $b_y$ & $1$ \\ \hline 
  Welch & unbiased 
  & $C_x b_x s_x^2 + C_y b_y s_y^2$
  & $C_x b_x \sigma_x^2 + C_y b_y \sigma_y^2$
  & $C_x b_x$ & $C_y b_y$ 
  & $\displaystyle{\frac{b_x \sigma_x^2 + b_y \sigma_y^2}
                           {C_x b_x \sigma_x^2 + C_y b_y \sigma_y^2}}$ \\ \hline 
  ordinary & pooled unbiased 
  & $C_{xy} (b_x + b_y) s_{xy}^2$
  & $C_{xy} (b_x + b_y) \sigma^2$
  & $C_{xy} b_y$ & $C_{xy} b_x$ 
  & $\displaystyle{\frac{1}{C_{xy}}}$ \\ \hline
  moderated & posterior 
  & $\displaystyle{\frac{(b_x + b_y)(d_0 s_0^2 + C_{xy} d_g s_{xy}^2)}
                                    {d_0 + d_g}}$
  & $\displaystyle{\frac{(b_x + b_y)(d_0 s_0^2 + C_{xy}d_g\sigma^2)}
                                    {d_0 + d_g}}$
  & $\displaystyle{\frac{C_{xy} d_g b_y}{d_0 + d_g}}$ 
  & $\displaystyle{\frac{C_{xy} d_g b_x}{d_0 + d_g}}$ 
  & $\displaystyle{\frac{d_0 + d_g}
                                    {d_0 s_0^2/\sigma^2 + C_{xy}d_g}}$ \\ \hline
  \end{tabular}
 \caption{Two-sample statistics}
 \label{tbl:twosmp}
\end{table}

\subsection{Special case: one-sample ordinary $t$} \label{sec:short}

When the scaling factor is na\"ive biased variance estimator $s_b^2$,
classic EE coincide with AEE ($r^2 = 1$). Polynomials representing a $4$-term
expansion are written in terms of $\lambda_j = \kappa_j/\sigma^j$, which
allows for a comparison with traditional expressions for standardized
mean.

\begin{flalign*} 
 q_1(x; 1) &= \, \frac{1}{6} \, \lambda_{3} \, {\left(2 \, x^{2} + 1\right)} &
\end{flalign*} 

\begin{flalign*}  
 q_2(x; 1) &= \, \frac{1}{12} \, \lambda_4 \, {\left(x^{3} -
3 \, x\right)}  -\frac{1}{18} \, \lambda_3^2 \, {\left(x^{5} + 2 \, x^{3} - 3 \, x\right)}
- \frac{1}{4} \, {\left( x^{3} + 3 \, x \right)} &
\end{flalign*} 

\begin{flalign*} 
 q_3(x; 1) = & \, - \frac{1}{40} \, \lambda_{5} \, {\left(2 \, x^{4} + 8 \,
     x^{2} + 1\right)} - \frac{1}{144} \, \lambda_{3} \lambda_{4} \, {\left(4 \,
x^{6} - 30 \, x^{4} - 90 \, x^{2} - 15\right)} \notag \\[1mm]
 & \, + \frac{1}{1296} \, \lambda_{3}^{3} \, {\left(8 \, x^{8} + 28 \,
     x^{6} - 210 \, x^{4} - 525 \, x^{2} - 105\right)} \notag \\[1mm]
 & \, + \frac{1}{24} \, \lambda_{3} \, {\left(2 \, x^{6} - 3 \, x^{4}
     - 6 \, x^{2}\right)} & 
\end{flalign*}

\begin{flalign*} 
 q_4(x; 1) = \, 
&- \frac{1}{90} \, \lambda_{6} \, {\left(2 \, x^{5} - 5 \, x^{3} - 15 \, x\right)}
+ \frac{1}{60} \, \lambda_{3} \lambda_{5} \, {\left(x^{7} + 8 \, x^{5} -
5 \, x^{3} - 30 \, x\right)} \notag \\[1mm]
&- \frac{1}{288} \, \lambda_{4}^{2} \, {\left(x^{7} - 21 \, x^{5} + 33
    \, x^{3} + 111 \, x\right)} \notag \\[1mm] 
&+ \frac{1}{216} \, \lambda_{3}^{2} \lambda_{4} \, {\left(x^{9} - 12
    \, x^{7} - 90 \, x^{5} + 36 \, x^{3} + 261 \, x\right)} \notag \\[1mm]
&- \frac{1}{1944} \, \lambda_{3}^{4} \, {\left(x^{11} + 5 \, x^{9} -
    90 \, x^{7} - 450 \, x^{5} + 45 \, x^{3} + 945 \, x\right)} \notag
\\[1mm] 
& + \frac{1}{48} \, \lambda_{4} \, {\left( x^{7} - 7 \, x^{5} + 9 \, x^{3} + 21 \,
    x\right)} - \frac{1}{72} \, \lambda_{3}^{2} \, {\left(x^{9} - 6 \,
    x^{7} - 12 \, x^{5} - 18 \, x^{3} - 9 \, x\right)} \notag \\[1mm] 
& - \frac{1}{96} \, { \left( 3 \, x^7 + 5 \, x^5 + 7 \, x^3 + 21 \, x \right) } &
\end{flalign*}

These expressions can be also used for the most common $t$-statistic
with unbiased variance estimator (if $X$ is distributed normally, this
statistic has Student's $t$-distribution). In that particular case,
$q_k(x; r) = q_k(x/r; 1) = q_k(\sqrt{\frac{n}{n -1}} \, x; 1)$. \\

\section{Illustration of Higher-Order
  Approximations} \label{sec:illustr}

To provide an illustration for higher-order approximations to the
distribution of a $t$-statistic, we consider an example with a small
sample ($n = 10$) of i.i.d. centered gamma distributed random
variables with shape parameter $k = 3$: $X \sim \Gamma(3, 1) - 3$ and
two versions of an ordinary $t$-statistic - with biased $s_b^2$ and
unbiased $s_{unb}^2$ variance estimators: $t_1 = n^{1/2}\bar{X}/s_b$
and $t_2 = n^{1/2}\bar{X}/s_{unb}$. Figure \ref{fig:terms} displays AEE
of up to fifth order ($0-4$-term expansions) for $t_1$ and $t_2$ along
with their respective true sampling distributions; known values of
$\lambda_j$ are used for the expansions. \\
\begin{figure}
 \centering
 \includegraphics[scale = 0.61]{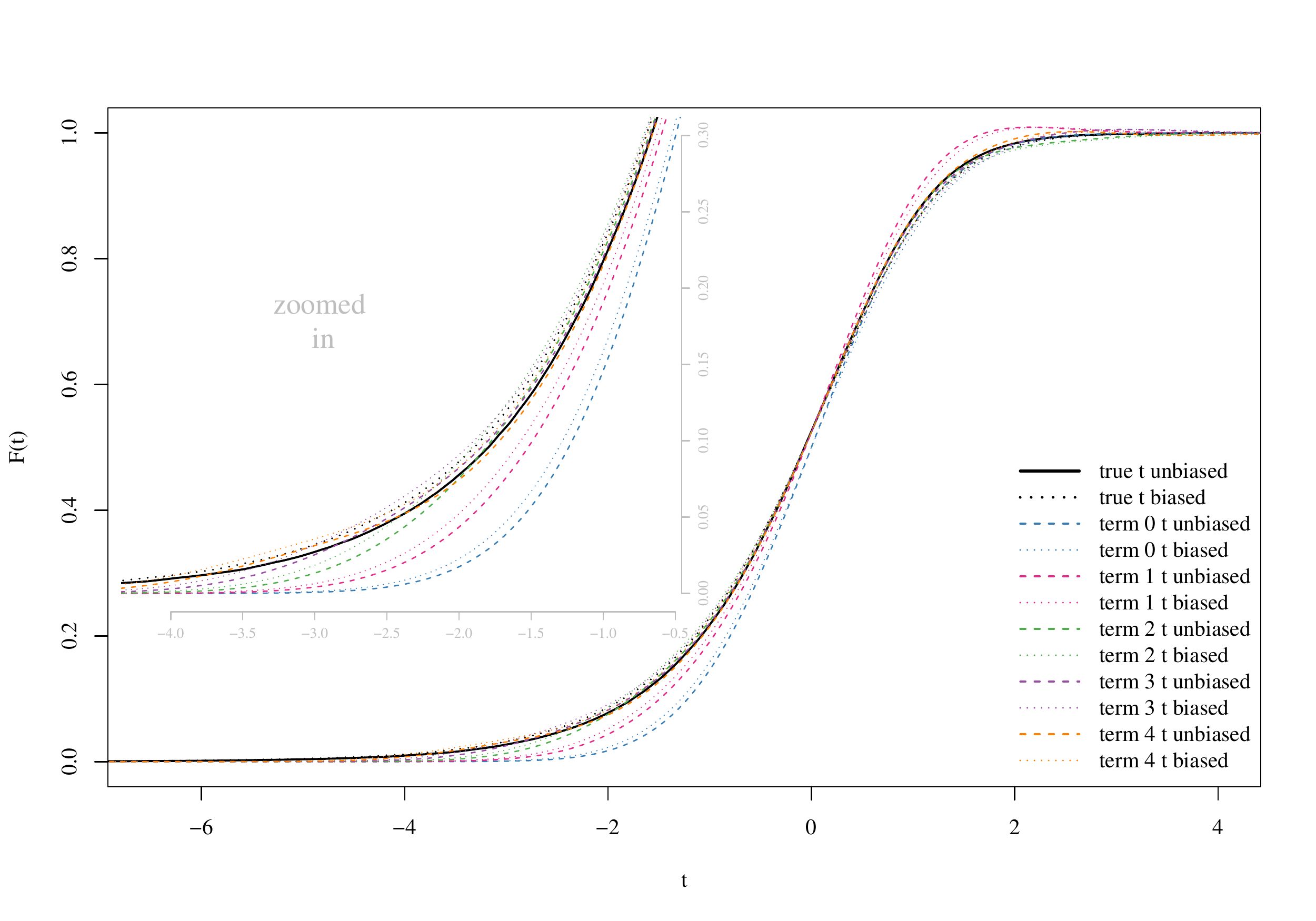} 
 \caption{Sampling distributions of an ordinary $t$-statistic with
   biased and unbiased variance estimators and their AEE
   approximations. $X \sim \Gamma(3, 1) - 3$, $n = 10$.} 
 \label{fig:terms}
\end{figure}

Edgeworth expansions are not probability functions and do not have
their properties - they are not necessarily monotonic everywhere and
might not be bounded by $0$ and $1$. This irregular behavior is
usually localized in the thinner tail of the distribution and
therefore EE are not very helpful there; it is clearly seen in the
second order approximation (term $1$) in the graph. We focus on the
thicker left tail where inference based on the first order
approximation would be anti-conservative (discussed in more detail
in Section \ref{sec:discussion}). The difference between the normal
approximation (term $0$) and the true distribution is quite striking;
subsequent orders improve approximation considerably. It appears that
the third order is already fairly close to the truth; however, as we
move away from the center and into the far tail, this approximation
gets further from the distribution and higher order terms come into
play proving the value of Edgeworth expansions of the orders beyond the
second and even third. This indicates how AEE can be used to adjust
inference for detected departures from normality in the tails of a
sampling distribution.  \\

\section{Discussion} \label{sec:discussion}

Generalized results for one- and two-sample $t$-statistics offer a 
possibility of using AEE in a variety of data analysis scenarios and
statistical procedures. As Figure \ref{fig:terms} demonstrates, first
order approximation may result in anti-conservative inference and
consequently lack of error rate  control
\cite{gerlovina2017opinion}. These issues arise when the tails of a
sampling distribution are thick, which is where EE behave nicely
providing increasingly closer approximations. Conversely,
non-monotoniciy and values beyond $[0, 1]$ discussed in Section
\ref{sec:illustr} can occur in the thinner tails where traditional
first order approximation provides conservative inference and thus can
be reliably used. For practical applications, this EE tail behavior means 
that some kind of ``tail diagnostic'' would need to be performed in
order to determine a usable order of approximation for each side. In
fact, the ``irregularity'' can be approached with ``it's not a bug, it's a
feature'' attitude: if the sample is representative, AEE tail
diagnostic can provide information about sampling distribution, 
specifically on symmetry and tail thickness. Then, each subsequent
order can be guaranteed to be more conservative than the previous one
- for example, in the context of hypothesis testing, the null
hypothesis would be rejected with more certainty as the order
increases. Another issue to be considered when adapting AEE to data
analysis is that since the true central moments of the data generating
distribution are not known, they would be substituted with
estimates. As higher moments are more sensitive to the choice of
estimators (e.g. na\"ive biased vs unbiased), estimators' behavior and
its effect on the performance of higher-order inference would need to
be explored. \\ 

AEE for ordinary one-sample $t$-statistics (Section \ref{sec:short})
and their comparison with EE for a standardized mean capture some key
differences between standardized and studentized statistics and
underline important features of $t$-statistic's distribution. To get
some insight into these differences, we can turn to the Student's
$t$-distribution with $n - 1$ degrees of freedom, which was derived as
a distribution of a $t$-statistic for a sample of $n$ i.i.d. normally
distributed random variables. Its derivation relies on a specific
property unique to Gaussian distribution: independence of sample
mean and sample variance. Without normality, this is no longer
the case, which can be easily seen with asymmetric
distributions. Consider a distribution of $X$ that is skewed to the    
right, with the thin left and thick right tails. While the
distribution of standardized mean (scaled by a constant) is also
skewed to the right, the distribution of studentized mean (scaled
by a random variable) is, in contrast, skewed to the left (Fig
\ref{fig:compDistA}). The reason for the ``flip'' stems from the fact
that observations that contribute to a greater sample average, coming
from the thicker tail, have greater dispersion as well, thus resulting
in a smaller value for $t$-statistic. Moreover, as can be seen in Fig
\ref{fig:compDistB}, the difference between thicker and thinner tails
appears to be even more pronounced than that of a ratio with assumed
indepence (obtained with permutation/random pairings of averages and
standard errors from different samples). EE for studentized mean do
not assume independence of sample mean and sample variance; truncated
series approach the correct shape of the distribution (as seen in Figure
\ref{fig:terms}). \\
\begin{figure}
    \centering
    \begin{subfigure}[b]{0.49\textwidth} 
      \includegraphics[scale =
      0.46]{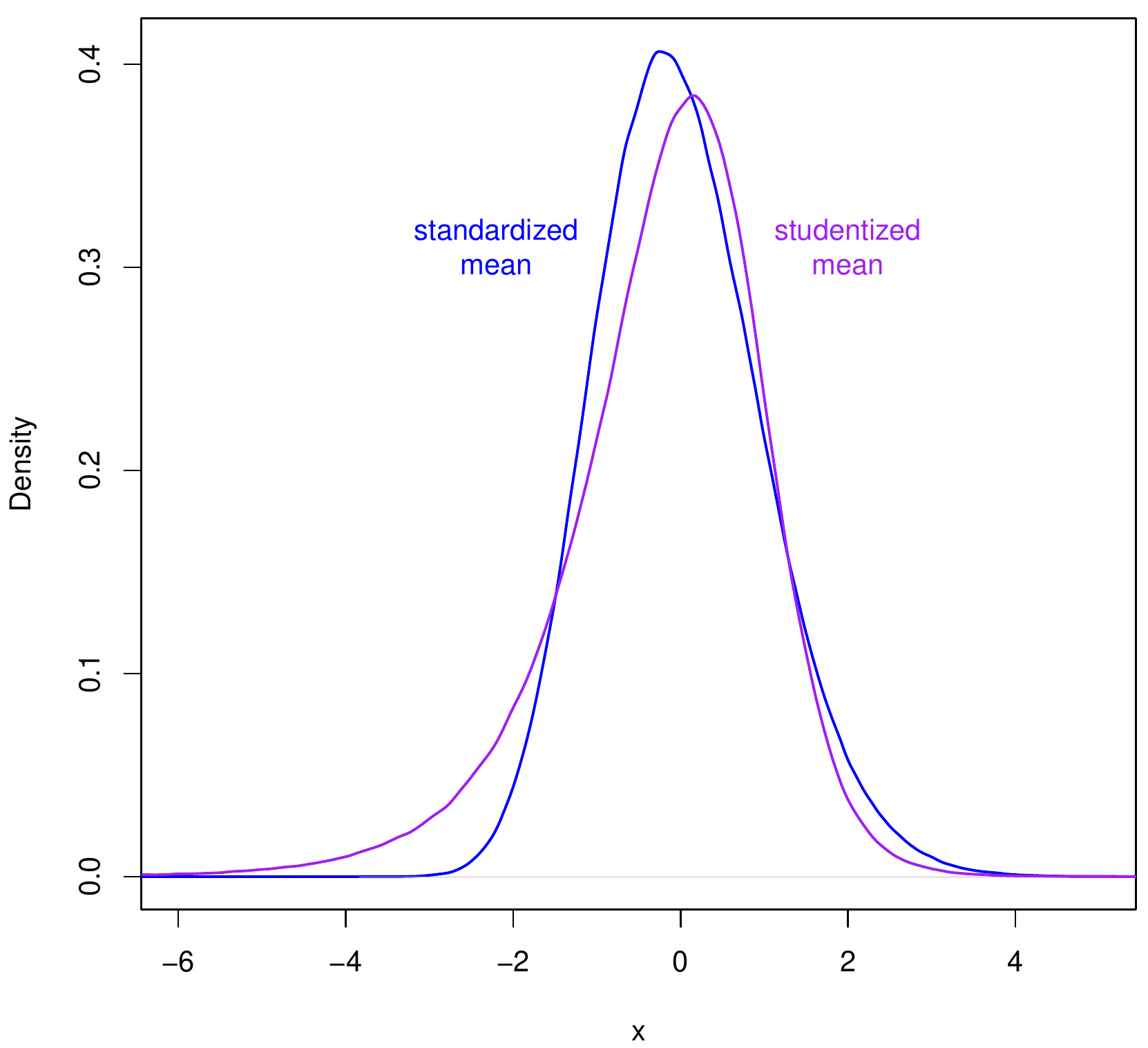} %
      \caption{Standardized mean $\bar{X}/s$ vs \\
                    studentized $\bar{X}/\sigma$} 
      \label{fig:compDistA}             
   \end{subfigure}%
   \begin{subfigure}[b]{0.49\textwidth}
      \includegraphics[scale = 0.46]{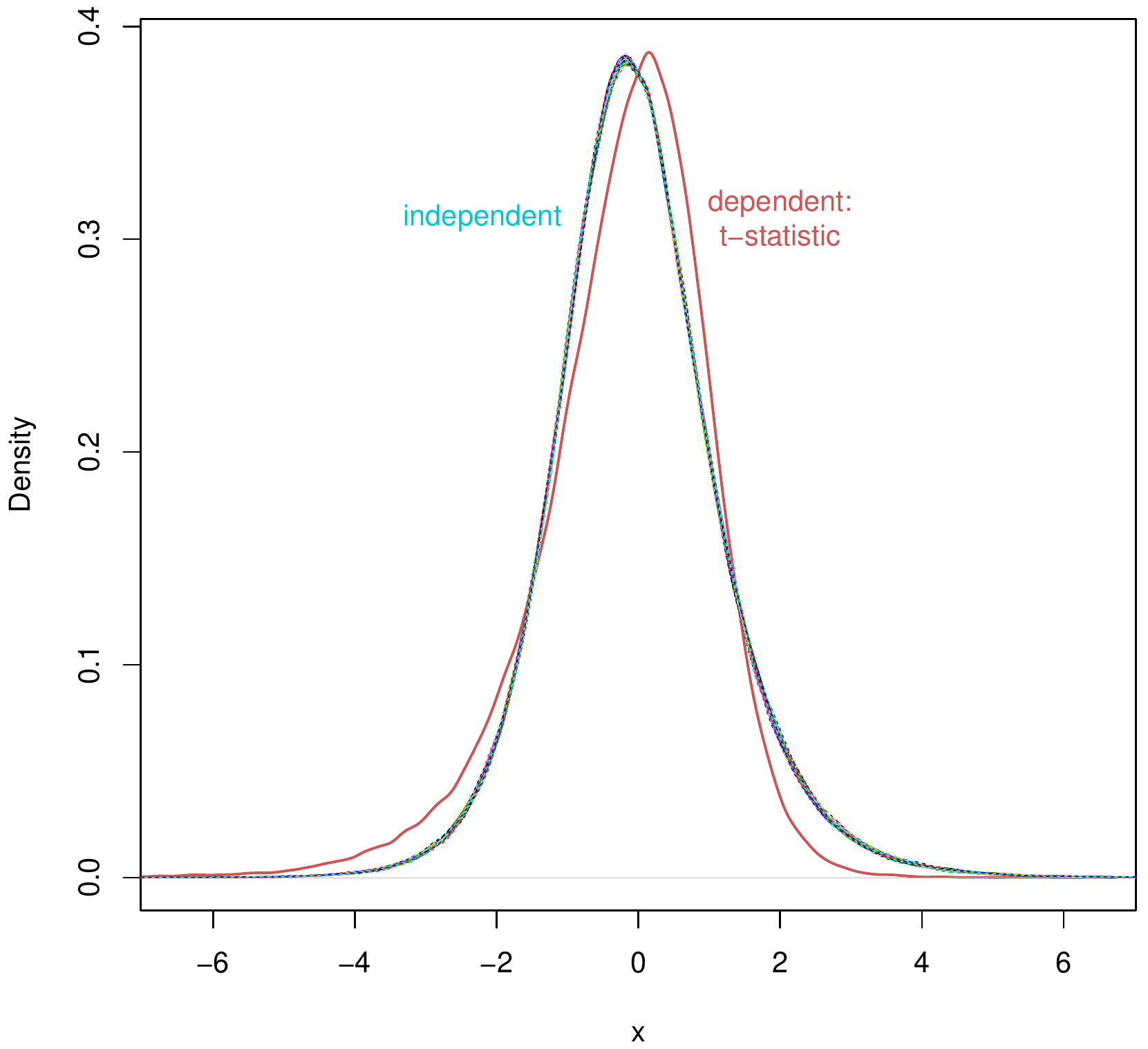} %
      \caption{$t$-statistic (dependent) vs \\
                    $\bar{X}/s$: $\bar{X} \perp s$}
     \label{fig:compDistB}%
  \end{subfigure}
  \caption{Distribution of scaled means, $X \sim \Gamma(3, 1) - 3$, $n
    = 10$.}   
\end{figure}

Another feature of these expansions in contrast with the ones for
standardized statistics is the cumulant order ``inconsistency'' inside
the polynomials for expansion terms. To see that, first consider a
standardized mean $\hat{\theta} = \sqrt{n}(\bar{X} - E(X))/\sigma$. For
cumulants $\kappa_{\hat{\theta}, j}$ of sampling distribution and
$\kappa_j$ of distribution of $X$, 
$\kappa_{\hat{\theta}, j} = n^{-\frac{j - 2}{2}} \kappa_j$ since
$\varphi_{\hat{\theta}}(t) = \big[\varphi (t/\sqrt{n}) \big]^n$,
  where $\varphi_{\hat{\theta}}$ and $\varphi$ are characteristic
  functions of $\hat{\theta}$ and $X$ respectively \cite{hall2013book}.
 The consequence of that is that standardized cumulants
  $\lambda_j$ are associated with $n^{-\frac{j - 2}{2}}$ and all the
  terms of EE polynomials respect that order - e.g. factors of the
  third term polynomial $q_3$ are $\lambda_5$, $\lambda_3 \lambda_4$,
  and $\lambda_3^3$. That cumulant relation is not true for a
  studentized mean, which is reflected in EE. Again, reference to
  normal distribution might provide some intuition for the effect of
  this difference. Consider $X \sim N(\mu, \sigma^2)$ and
$\hat{\theta} = \sqrt{n} (\bar{X} - \mu)/s \sim t_{n-1}$. Then
$\lambda_j = 0$, $j = 3, \dotsc$ and $q_1 = q_3 = \dotsb = 0$. Even
term polynomials ($q_2, q_4, \dotsc$) have remaining non-zero terms
that make the tails thicker, consistent with the fact that Student's
$t$-distribution has non-unit variance and thicker tails than
normal. For non-normal distributions, polynomial terms that 
contain cumulants but are not of a ``regular'' order are likely to also
contribute to thickness of the tails though it is harder to assess. \\

Student's $t$, while not a sampling distribution for any random
variable but normal (and not a limit distribution), can be useful in
exploring far tails of distributions of studentized mean-based
statistics and an effect of sample size on associated critical values
\cite{gerlovina2017opinion} in, for example, high-dimensional data
analysis with multiple comparisons. In fact, it is routinely used in
practice, with stated but not always warranted normality assumption to
justify its use; it can be argued that it still provides useful
approximation to sampling distribution for large deviations and small
sample size (\cite{zholud2014tail}). Combining higher-order inference
approach of AEE with $t$-distribution for challenging extreme tail
estimation could be another fruitful direction for achieving more
reliable inference. \\

\begin{appendix}

\section{Proofs of Propositions} \label{sec:proofProp}

\begin{proof}[Proof of Proposition \ref{prop:mom1}]
 \begin{flalign}  \label{eq:thetamt1}
  &\hat{\theta}^m = n^{\frac{m}{2}} A^{-\frac{m}{2}} \bar{X}^m \, ( 1 + \gamma_1 -
    \gamma_2 )^{-\frac{m}{2}} = n^{\frac{m}{2}} A^{-\frac{m}{2}} \bar{X}^m 
   \left[ 1 + \sum_{k=1}^\infty a_{m,k} (\gamma_1 - \gamma_2)^k
   \right], &
\end{flalign}
where $\gamma_1 = A^{-1}B \bar{X}_s = \mathcal{O} \left(n^{-1/2}
\right)$, $\gamma_2 = A^{-1} B \bar{X}^2 = \mathcal{O} \left(n^{-1}
\right)$, and $a_{m, k}$ as defined in \eqref{eq:amk}. \\

From Taylor expansion of $(1 + \gamma_1 - \gamma_2)^{-\frac{m}{2}}$
and, subsequently, from $(\gamma_1 - \gamma_2)^k$ \eqref{eq:thetamt1}
we only need the terms with factors up to $n^{-\frac{K}{2}}$. Knowing
the orders of $\gamma_1$ and $\gamma_2$ does not only allow us to use
Taylor expansion in the first place, it also provides a tool to keep
only the relevant terms of the expansion. \\ 

Start with grouping the terms by orders (powers of $n^{-\frac{1}{2}}$):
\begin{alignat*}{2} 
 {\left(1 + \gamma_1 - \gamma_2\right)}^{-\frac{m}{2}} 
   &= && \, 1 + \sum_{k = 1}^{\infty} a_{m, k} \sum_{i = 0}^k {k \choose
     i} (-1)^i \, \gamma_1^{k - i} \, \gamma_2^i \notag \\
   & = && \, \Bigg( 1 + \sum_{k=1}^{\infty} \Bigg[ a_{m,k} {k \choose 0} 
           \gamma_1^k \, \gamma_2^0 - a_{m, k-1} {k - 1 \choose 1}
           \gamma_1^{k-2} \, \gamma_2^1 + a_{m, k-2} {k - 2 \choose 2}
           \gamma_1^{k-4} \, \gamma_2^2 \nonumber  \\[2mm]  
  &   && \; \; \; \; - \dotsb + 
   \begin{cases} 
    \, a_{m, \frac{k}{2}} (-1)^{\frac{k}{2}} \, \displaystyle{{\frac{k}{2} \choose
        \frac{k}{2}}} \gamma_1^0 \, \gamma_2^{\frac{k}{2}} \, \Bigg]
    \Bigg) &\text{ - for even } k \\[1.5em] 
    \, a_{m, \frac{k+1}{2}} (-1)^{\frac{k-1}{2}} \displaystyle{{\frac{k+1}{2} \choose
        \frac{k-1}{2}}} \gamma_1^1 \,
    \gamma_2^{\frac{k-1}{2}} \, \Bigg] \Bigg) &\text{ - for odd } k
   \end{cases} \nonumber \\[2mm]
  & = && \, 1 + \sum_{k = 1}^{\infty} \sum_{i =
    0}^{\left\lfloor\frac{k}{2}\right\rfloor} a_{m, k - i} (-1)^i {k - i
    \choose i} \gamma_1^{k - 2i} \, \gamma_2^i  
\end{alignat*}

From this, we can pick $K$ terms and get
\begin{flalign*} 
 \hat{\theta}^m &= \, n^{\frac{m}{2}} A^{-\frac{m}{2}} \bar{X}^m \left[
   \, 1 + \sum_{k=1}^K \sum_{i=0}^{\left\lfloor \frac{k}{2} \right\rfloor} a_{m, k-i}
   (-1)^i {k-i \choose i} \gamma_1^{k-2i} \, \gamma_2^i \, \right] +
 \mathcal{O} \left( n^{-\frac{K + 1}{2}} \right) & 
\end{flalign*}

Then
\begin{equation*}
 E \left( \hat{\theta}^m \right) = n^{\frac{m}{2}} A^{-\frac{m}{2}} \Bigg[ \,
   \rho(m, 0) + \sum_{k=1}^K \sum_{i=0}^{\left\lfloor
       \frac{k}{2} \right\rfloor} a_{m, k-i} (-1)^i {k-i \choose i}
   \, \frac{B^{k-i}}{A^{k-i}} \, \rho(m+2i, k-2i) \, \Bigg] +
   \mathcal{O} \left( n^{-\frac{K + 1}{2}} \right)  
\end{equation*}
\end{proof}

\vspace{1em}

\begin{proof}[Proof of Proposition \ref{prop:mom2}]
Applying the same argument for truncation and leaving only terms of
relevant orders as in the one-sample case, we get the expression:
\begin{equation} \label{eq:thetamt2}
 \hat{\theta}^m = n^{\frac{m}{2}} A^{-\frac{m}{2}} (\bar{X} - \bar{Y})^m \left[ 1
   + \sum_{k=1}^K \sum_{i=0}^{\left\lfloor \frac{k}{2} \right\rfloor}
    (-1)^i a_{m, k-i} {k-i \choose i} \gamma_1^{k-2i} \, \gamma_2^i
  \right] +  \mathcal{O} \left( n^{-\frac{K + 1}{2}} \right)  
\end{equation}
where $\gamma_1 = A^{-1}(B_x\bar{X}_s + B_y\bar{Y}_s)$, $\gamma_2 =
A^{-1} (B_x\bar{X}^2 + B_y\bar{Y}^2)$, and $a_{m, k}$ is the same as 
in \eqref{eq:amk}. It is straightforward to show that $\gamma_1 =
\mathcal{O}(n^{-\frac{1}{2}})$ and $\gamma_2 = \mathcal{O}(n^{-1})$.\\

For the expectation, we need to expand $\gamma_1^k \, \gamma_2^l$:
\begin{equation*}
 \gamma_1^k \, \gamma_2^l = \frac{1}{A^{k + l}} \sum_{i=0}^k
 \sum_{j=0}^l {k \choose i} {l \choose j} B_x^{(k + l) - (i + j)}
 B_y^{i + j} \bar{X}^{2(l - j)} \bar{X}_s^{k - i} \, \bar{Y}^{2j} \, \bar{Y}_s^i. 
\end{equation*}
\begin{alignat*}{3} 
 E \left(\hat{\theta}^m \right) = & \, n^{\frac{m}{2}} A^{-\frac{m}{2}}
 \sum_{j=0}^m (-1)^j {m \choose j} \bigg[ \rho(m-j, 0) \, \tau(j, 0)
 + \sum_{k=1}^K \sum_{i=0}^{\left\lfloor \frac{k}{2} \right\rfloor} 
 (-1)^i a_{m, k - i} {k-i \choose i} A^{i-k} \notag \\
 & \phantom{+} \; \times \sum_{u=0}^{k-2i} \sum_{v=0}^i {k - 2i
   \choose u} \! {i \choose v} B_x^{(k - i) - (u + v)} B_y^{u + v} 
 \rho\big(m-j+2(i-v), k-2i-u\big) \, \tau(j+2v, u) \bigg] \notag \\
 & +  \mathcal{O} \left( n^{-\frac{K + 1}{2}} \right)  
\end{alignat*}
\end{proof}

\section{Proofs of Theorems} \label{sec:proofThm}

\begin{proof}[Proof of Theorem \ref{thm:AEE1}]
  We can write $\hat{\theta}_n$ in the following way:
\begin{equation*}
 \hat{\theta}_n = n^{\frac{1}{2}} \bar{X} \, A^{-\frac{1}{2}} \, (1 +
 b\gamma)^{-\frac{1}{2}}, \text{ where} 
\end{equation*}
 $b = B/A$, $\gamma = \bar{X}_s - \bar{X}^2$, $\bar{X} = n^{-1}
 \sum_{i = 1}^n X_i$, and $\bar{X}_s = n^{-1}\sum_{i = 1}^n X_i^2 -
 \mu_2$. 
\begin{equation*}
 \hat{\theta}_n^m = n^{\frac{m}{2}} \bar{X}^m A^{-\frac{m}{2}} (1 + b
  \gamma)^{-\frac{m}{2}} =   n^{\frac{m}{2}} A^{-\frac{m}{2}} \sum_{k
    = 0}^{\infty} b^k a_{m, k} \sum_{i = 0}^k {k \choose i} (-1)^i
  \bar{X}_s^{k - i} \bar{X}^{2i + m}, 
\end{equation*}
where $a_{m, k}$ is the same as in \eqref{eq:amk}. Taking expectation,
we obtain
\begin{equation*}
 \mu_{\hat{\theta}, m} = E \left(\hat{\theta}_n^m \right) = n^{\frac{m}{2}}
  A^{-\frac{m}{2}} \sum_{k = 0}^{\infty} b^k a_{m, k} \sum_{i = 0}^k
  {k \choose i} (-1)^i \rho(2i + m, k - i).
\end{equation*}
It can be shown that $\rho(u, w) = \sum_{v = \left\lceil\frac{u +
      w}{2} \right\rceil}^{u + w - 1} \frac{1}{n^v} \beta(u, w, v)$, 
where $\beta(u, w, v)$ does not depend on $n$ and only depends on
moments of $X$. Then
\begin{equation*}
 \mu_{\hat{\theta}, m} = n^{\frac{m}{2}} A^{-\frac{m}{2}} \sum_{k = 0}^{\infty} b^k a_{m, k}
    \sum_{i = 0}^k {k \choose i} (-1)^i   \sum_{v = \left\lceil
    \frac{m + k + i}{2} \right \rceil}^{m + k + i - 1} \frac{1}{n^v}
    \beta(2i + m, k - i, v).  
\end{equation*}

Switch the order of summation, summing over $v$ first and over 
$k$ second:
\begin{equation*}
 \mu_{\hat{\theta}, m} = A^{-\frac{m}{2}} n^{-\frac{1}{2} \delta(m)} \sum_{v
   = 0}^{\infty} \frac{1}{n^v} \sum_{k = k_1}^{k_2} b^k \tilde{g}(k,
 v; \, m),
  \end{equation*}
where $\delta(m) = m \mod 2$, $k_1 = max\left(0, \left \lceil \frac{1}{2}
  \left(v - \left \lfloor \frac{m}{2} \right \rfloor + 1\right)
\right \rceil \right)$, $k_2 = 2v + \delta(m)$, $\tilde{g}(k, v; \, m) = a_{m,
k} \sum_{i = i_1}^{i_2} {k \choose i} (-1)^i \beta\big(2i + m, k - i,
v + \left \lceil \frac{m}{2} \right \rceil \big)$ with $i_1 =
max\big(0, v - \left\lfloor \frac{m}{2} \right\rfloor + 1 - k\big)$ and
$i_2 = min(k, 2v + \delta(m) - k)$; $\tilde{g}(k, v; \, m)$ does not depend on
$n$. If $k_1 > 0$, we can sum over $k$ starting from $k = 0$ and set
$\tilde{g}(k, v; \, m) = 0$ if $0 \leqslant k < k_1$.  \\

As we only need a finite number of terms for the expansions, we
consider the finite sum in the moments as well:
\begin{equation*}
 \mu_{\hat{\theta}, m} 
  = A^{-\frac{m}{2}} n^{-\frac{1}{2} \delta(m)} \left[\sum_{v = 0}^{V}
    \frac{1}{n^v} \sum_{k = 0}^{2v + \delta(m)} b^k \, \tilde{g}(k, v;
    \, m) + \mathcal{O} \left(n^{-(V + 1)} \right) \right].
 \end{equation*}

Next we look at the moment products. Using an induction-like argument,
we can show that 
\begin{align} \label{eq:momprod}
  \mu_{\hat{\theta}, m_1}^{l_1}\mu_{\hat{\theta}, m_2}^{l_2} \dotsm \mu_{\hat{\theta},
    m_d}^{l_d}= & \, 
  A^{-\frac{1}{2} \sum_{i = 1}^d l_i m_i} n^{-\frac{1}{2} \sum_{i =
      1}^d l_i \delta(m_i)} \notag \\
   & \times \left[ \sum_{v = 0}^{V} \frac{1}{n^v}
          \sum_{k = 0}^{2v + \sum_{i = 1}^d l_i\delta(m_i)} b^k
          g(k, v; \, \boldsymbol{m}, \boldsymbol{l}) 
  + \mathcal{O} \left(n^{-(V + 1)}\right) \right], 
\end{align}
where $\boldsymbol{m} = (m_1, \dotsc, m_d)$, $\boldsymbol{l} = (l_1,
\dotsc, l_d)$, and $g(k, v; \, \boldsymbol{m}, \boldsymbol{l})$ does
not depend on $n$. Indeed, the base case is $\boldsymbol{m} = (m)$,
$\boldsymbol{l} = (1)$; then $g(k, v; \, \boldsymbol{m},
\boldsymbol{l}) = \tilde{g}(k, v; \, m)$ does not depend on $n$. Next,
consider $\boldsymbol{m} = (m_1, \dotsc, m_d)$, $\boldsymbol{l} = (l_1,
\dotsc, l_d)$, $\boldsymbol{k} = (k_1, \dotsc, k_g)$, $\boldsymbol{j}
= (j_1, \dotsc, j_g)$, and $g(k, v; \, \boldsymbol{m},
\boldsymbol{l})$, $g(k, v; \, \boldsymbol{k}, \boldsymbol{j})$ that do
not depend on $n$ and find $\mu_{\hat{\theta}, m_1}^{l_1} \dotsm \mu_{\hat{\theta},
  m_d}^{l_d} \mu_{\hat{\theta}, k_1}^{j_1} \dotsm \mu_{\hat{\theta},
  k_g}^{j_g}$. By simple multiplication, gathering the terms by powers
of $1/n$, and adding a finite number of resulting higher-order terms
to $\mathcal{O} \left(n^{-(V + 1)} \right)$, we get
\begin{align*} 
  \mu_{\hat{\theta}, m_1}^{l_1} \dotsm \mu_{\hat{\theta}, m_d}^{l_d}\mu_{\hat{\theta},
  k_1}^{j_1} \dotsm \mu_{\hat{\theta}, k_g}^{j_g}= & \, A^{-\frac{1}{2}
                  \left(\sum_{i = 1}^d l_i m_i + \sum_{i = 1}^g j_i k_i\right)}
 n^{-\frac{1}{2} \left(\sum_{i = 1}^d l_i \delta(m_i) + \sum_{i = 1}^g
                                               j_i \delta(k_i) \right)} \notag \\
   & \times \left[ \sum_{v = 0}^{V} \frac{1}{n^v}
          \sum_{k = 0}^{2v + \sum_{i = 1}^d l_i\delta(m_i) + \sum_{i =
     1}^g j_i\delta(k_i)} b^k g(k, v; \, \boldsymbol{m_s}, \boldsymbol{l_s}) 
  + \mathcal{O} \left(n^{-(V + 1)}\right) \right], 
\end{align*}
where $\boldsymbol{m_s} = (m_1, \dotsc, m_d, k_1, \dotsc, k_g)$,
$\boldsymbol{l_s} = (l_1, \dotsc, l_d, j_1, \dotsc, j_g)$, and
\begin{equation*}
  g(k, v; \, \boldsymbol{m_s}, \boldsymbol{l_s}) =
  \sum_{u = 0}^v \sum_{l = 0}^{min\left[k, 2u + \sum_{i = 1}^d l_i
      \delta(m_i) \right]} g(l, u; \, \boldsymbol{m}, \boldsymbol{l})
  g(k - l, v - u; \, \boldsymbol{k}, \boldsymbol{j}).
\end{equation*}
Thus $g(k, v; \, \boldsymbol{m_s}, \boldsymbol{l_s})$ does not depend on
$n$ and equation \eqref{eq:momprod} is true. \\ 

The next step is obtaining expressions for cumulants. Let $M$ be the
cumulant's order. The cumulant $\kappa_{\hat{\theta}, M}$ is written as a
sum of moment products with their respective coefficients: 
\begin{equation*}
 \kappa_{\hat{\theta}, M} = \sum_{j = 1}^{J_M} C(M,j) \mu_{\hat{\theta}, 1}^{l(M,
   j, 1)} \mu_{\hat{\theta}, 2}^{l(M, j, 2)} \dotsm \mu_{\hat{\theta}, M}^{l(M, j,
   M)}
\end{equation*}
 with a condition $\sum_{m = 1}^M m \, l(M, j, m) = M$ for all
 $j$; $l(M, j, m)$ are non-negative integers. Plugging in the
 expressions for products of moments, we get 
\begin{align*}
 \kappa_{\hat{\theta}, M} =& \sum_{j = 1}^{J_M} C(M, j) A^{-\frac{1}{2}
   \sum_{m = 1}^M m \, l(M, j, m)} n^{-\frac{1}{2} \sum_{m = 1}^M
   \delta(m) l(M, j, m)} \notag \\
 & \times \left[\sum_{v = 0}^{V} \frac{1}{n^v} \sum_{k = 0}^{2v +
   \sum_{m = 1}^M \delta(m) l(M, j, m)} b^k g(k, v; \, \boldsymbol{m},
   \boldsymbol{l}) + \mathcal{O} \left(n^{-(V + 1)} \right) \right],
\end{align*}
where $\boldsymbol{m} = (1, \dotsc, M)$ and $\boldsymbol{l} = (l(M, j,
1), \dotsc, l(M, j, M))$. It can be shown that
$\sum_{m = 1}^M \delta(m) \, l(M, j, m)$ can be expressed as
$\delta(M) + 2L_1(M, j)$, where $L_1(M, j)$ is a non-negative
integer. Then   
\begin{align} \label{eq:proof01}
 \kappa_{\hat{\theta}, M} =& \sum_{j = 1}^{J_M} C(M, j) A^{-\frac{M}{2}}
                      n^{-\frac{1}{2} \delta(M)} \notag \\
 & \times \left[\sum_{v = 0}^{V} \frac{1}{n^{v + L_1(M, j)}} \sum_{k = 0}^{2v +
   2L_1(M, j) + \delta(M)} b^k g(k, v; \, \boldsymbol{m},
   \boldsymbol{l}) + \mathcal{O} \left(n^{-(V + 1)} \right) \right].
\end{align}
Let $w = v + L_1(M, j)$. Then we can write:
\begin{align} \label{eq:proof02}
 \kappa_{\hat{\theta}, M} =& \sum_{j = 1}^{J_M} C(M, j) A^{-\frac{M}{2}}
                      n^{-\frac{1}{2} \delta(M)} \notag \\
 & \times \left[\sum_{w = 0}^{V} \frac{1}{n^{w}} \sum_{k = 0}^{2w +
   \delta(M)} b^k g(k, w - L_1(M, j); \, \boldsymbol{m},
   \boldsymbol{l}) + \mathcal{O} \left(n^{-(V + 1)} \right) \right].
\end{align}
Note that if $v + L_1(M, j) > V$, some terms $v + L_1(M, j)$ are
included in $\mathcal{O} \left(n^{-(V + 1)} \right)$ in equation
\eqref{eq:proof02}, and therefore $\mathcal{O} \left(n^{-(V + 1)}
\right)$ in \eqref{eq:proof02} has additional terms compared to
that of equation \eqref{eq:proof01}. \\

Now, changing the order of summation,
\begin{align*} 
 \kappa_{\hat{\theta}, M} =& A^{-\frac{M}{2}} \, n^{-\frac{1}{2} \delta(M)} \notag \\
 & \times \left[\sum_{w = 0}^{V} \frac{1}{n^{w}} \sum_{k = 0}^{2w +
   \delta(M)} b^k \sum_{j = 1}^{J_M} C(M, j) g(k, w - L_1(M, j); \,
   \boldsymbol{m}, \boldsymbol{l}) + \mathcal{O}
   \left(n^{-(V + 1)} \right) \right].
\end{align*}

Let $G(M, k, w)$ denote $\sum_{j = 1}^{J_M} C(M, j) g(k, w - L_1(M,
j); \, \boldsymbol{m}, \boldsymbol{l} \Big)$, which does not
   depend on $n$. Then
\begin{equation} \label{eq:woutn}
 \kappa_{\hat{\theta}, M} = n^{-\frac{1}{2} \delta(M)} 
 \left[\sum_{w= \left\lfloor \frac{M}{2} \right\rfloor - 1 }^{V}
   \frac{1}{n^{w}} \sum_{k = 0}^{2w + \delta(M)} A^{-\frac{M}{2}} b^k \,
  G(M, k, w) + \mathcal{O} \left(n^{-(V + 1)} \right) \right].
\end{equation}
Note that summation over $w$ starts with $w = \left\lfloor \frac{M}{2}
\right\rfloor - 1$ and not with $w = 0$ (see \cite{hall2013book} Theorem 2.1). \\

Now we turn to the case where $A$ and $b$ depend on $n$. 
Let $b^k \, A^{-\frac{M}{2}} = \sum_{i = 0}^{\infty} \frac{1}{n^i} \,
t(M, k, i) = \sum_{i = 0}^V \frac{1}{n^i} \, t(M, k, i) + \mathcal{O}
\left(n^{-(V + 1)} \right)$, where $t(M, k, i)$ is of the order
$\mathcal{O}(1)$ and does not depend on $n$.
\begin{equation*}
 \sum_{k = 0}^{2w + \delta(M)} A^{-\frac{M}{2}} \, b^k \, G(M,
  k, w) = \sum_{i = 0}^V \frac{1}{n^i} \sum_{k = 0}^{2w + \delta(M)}
  G(M, k, w) \, t(M, k, i) + \mathcal{O} \left(n^{-(V +
    1)} \right)
\end{equation*}
Then
\begin{equation*}
  \kappa_{\hat{\theta}, M} = n^{-\frac{1}{2} \delta(M)}
    \left[ \sum_{u= \left\lfloor \frac{M}{2} \right\rfloor - 1 }^{V}
    \frac{1}{n^u} \, \tilde{t}(M, u) +  \mathcal{O} \left(n^{-(V +
    1)}\right) \right],
\end{equation*}
where $\tilde{t}(M, u) = \sum_{i = 0}^{u - \left\lfloor \frac{M}{2}
  \right\rfloor + 1} \sum_{k = 0}^{2u - 2i + \delta(M)} G(M,
k, u - i) \, t(M,k, i)$ does not depend on $n$. \\

Thus,
\begin{equation} \label{eq:withn}
 \kappa_{\hat{\theta}, M} = n^{-\frac{M - 2}{2}} \sum_{u = 0}^{V -
   \left\lfloor \frac{M}{2} \right\rfloor + 1} \frac{1}{n^u} \, 
   \tilde{t}\Big(M, u + \left\lfloor \frac{M}{2} \right\rfloor - 1
   \Big) + \mathcal{O} \left( n^{-\big(V + 1 + \frac{1}{2} \delta(M)
   \big)} \right)
\end{equation}
This expression corresponds to the equation $(2.20)$, Chapter $2.2$ in
\cite{hall2013book}: 
\begin{equation*}
  \kappa_{\hat{\theta}, M} = n^{-\frac{M - 2}{2}} \Big(k_{M, 0} +
  \frac{1}{n} k_{M, 1} + \frac{1}{n^2} k_{M, 2} + \dotsb \Big), 
\end{equation*}
where it is shown that in this case Edgeworth expansion is valid. \\

Finite-term difference between representations of $\kappa_{\hat{\theta}, M}$
in equations \eqref{eq:woutn} and \eqref{eq:withn} is
$\mathcal{O} \left(n^{-(V + 1) - \frac{1}{2} \delta(M)} \right)$; it
is easy to see that $V = \left\lfloor \frac{K}{2} \right\rfloor$ and
$V + \frac{1}{2}\delta(K) = \frac{K}{2}$. The difference of
$\mathcal{O}\big(n^{-(\frac{K}{2} + 1)}\big)$ in cumulants translates into
    the difference of $\mathcal{O}\big(n^{-\frac{K + 1}{2}} \big)$ between
    corresponding K-term Edgeworth expansions, and thus the
    expansion is also valid when $A$ and $b$ depend on $n$. \\
\end{proof}

\begin{proof}[Proof of Theorem \ref{thm:AEE2}]
  Comparing equations \eqref{eq:thetamt1} and \eqref{eq:thetamt2} from
  the proofs of Propositions 1 and 2, we can see that the general
  moment structures of generalized one- and two-sample $t$-statistics
  are similar, including the order of $\gamma_1$ and
  $\gamma_2$. Therefore the proof of Theorem 2 follows the same steps
  as the proof of Theorem 1. 
\end{proof}

\end{appendix}

 \section*{Acknowledgements}
 We are deeply grateful to Boris Gerlovin for his invaluable ideas,
 suggestions, and our productive discussions. \\
 
 This project was supported by the National Institute of Environmental
 Health Sciences [P42ES004705] Superfund Research Program at UC
 Berkeley. 

\bibliographystyle{ieeetr}
\bibliography{EdgeTheory}

\end{document}